\documentclass{amsart}
\usepackage{amssymb}
\usepackage[dvips]{epsfig}
\usepackage{graphicx}
\usepackage{color}

 \usepackage[latin1]{inputenc}
 \usepackage[T1]{fontenc}
 \usepackage[normalem]{ulem}
\usepackage{verbatim}
 \usepackage{graphicx}
\usepackage[latin1]{inputenc}
\usepackage{amsmath}
\usepackage{amssymb}
\usepackage{amsfonts}
\usepackage{amsthm}
\usepackage{bbm}
\usepackage{hyperref}

\newcommand{\R}{\mathbb R}
\newcommand{\p}{\partial}

\newcommand{\Z}{\mathbb Z}

\newcommand{\g}{\gamma}

\newcommand{\ji}{\langle}
\newcommand{\jd}{\rangle}

\numberwithin{equation}{section}

\newtheorem{theorem}{Theorem}[section]
\newtheorem{proposition}[theorem]{Proposition}
\newtheorem{remark}{Remark}[section]
\newtheorem{lemma}[theorem]{Lemma}
\newtheorem{corollary}[theorem]{Corollary}

\newtheorem*{TA}{Theorem A}
\newtheorem*{TB}{Theorem B}
\newtheorem*{TC}{Theorem C}

\begin{document}

\title[Zakharov-Kuznetsov equation]{Dispersive Blow-up for Solutions of the Zakharov-Kuznetsov equation\\ }
\author[F. Linares]{F. Linares}
\address{ IMPA\\ Estrada Dona Castorina 110\\ Rio de Janeiro 22460-320, RJ Brasil}
\email{linares@impa.br}
\author[A. Pastor]{A. Pastor}
\address{IMECC-UNICAMP, Rua S\'ergio Buarque de Holanda, 651, 13083-859, Campinas-SP, Brazil }
\email{apastor@ime.unicamp.br}
\author[J. Drumond Silva]{J. Drumond Silva}
\address{Center for Mathematical Analysis, Geometry and Dynamical Systems,Department of Mathematics, Instituto Superior T\'ecnico, Universidade de Lisboa, Av. Rovisco Pais, 1049-001 Lisboa, Portugal}
\email{jsilva@math.tecnico.ulisboa.pt }


\begin{abstract}   The main purpose here is the study of dispersive blow-up for solutions of the Zakharov-Kuznetsov equation.
	Dispersive blow-up refers to point singularities due to the focusing of short or long waves.  We will construct initial data such 
	that solutions of the linear problem present this kind of singularities. Then we show that the
	corresponding solutions of the nonlinear problem 
	present dispersive blow-up inherited from the linear component part of the equation. Similar results are obtained for the generalized Zakharov-Kuznetsov equation.
\end{abstract}

\maketitle


\section{Introduction}

In this paper we consider solutions of the initial value problem (IVP) associa\-ted to the two dimensional Zakharov-Kuznetsov (ZK) and generalized Zakharov-Kuznetsov (gZK) equations, respectively $k=1$ and $k\geq2$,
\begin{equation}\label{zk-2d}
	\begin{cases}
		\p_tu+\p_x\Delta u+u^k\p_xu=0, \hskip10pt (x,y)\in \R^2, \, t\in\R,  k\in\Z^{+},\\
		u(x, y,0)=u_0(x,y)
	\end{cases}
\end{equation}
where $u$ is a real function and $\Delta$ denotes the Laplace operator in space variables.

The equation above, for $k=1$, arises in the  context of plasma physics, where it was formally derived by Zakharov and Kuznetsov
\cite{KZ}  as a long wave small-amplitude limit of the Euler-Poisson system in the ``cold plasma''
approximation. This formal long-wave limit was rigorously justified by  Lannes, Linares and Saut   in \cite{LLS}  (see also \cite{H} for
derivation in a different context).

The Zakharov-Kuznetsov equation can be seen as a natural multi-dimensional extension of the Korteweg-de Vries (KdV) equation, quite 
different from the well-known Kadomtsev-Petviashvili (KP) equation which is obtained as an asymptotic model of various nonlinear dispersive
systems under a different scaling.

Contrary to the Korteweg-de Vries or the Kadomtsev-Petviashvili equations, the Za\-kha\-rov-Kuznetsov equation is not completely integrable but it has a Hamiltonian structure and possesses three invariants, namely,
\begin{equation*}
	I(t)=\int_{\R^2} u(x,y,t)\,dxdy= \int_{\R^2} u_0(x,y)\,dxdy=I(0),
\end{equation*}
\begin{equation*}
	M(t)=\int_{\R^2}u^2(x,y,t)\,dxdy=\int_{\R^2} u_0^2(x,y)\,dxdy= M(0),
\end{equation*}
and
\begin{equation*}
	E(t)=\frac12\int_{\R^2}\Big(|\nabla u|^2-\frac{u^3}{3}\Big)dxdy=\frac12\int_{\R^2}\Big(|\nabla u_0|^2-\frac{u_0^3}{3}\Big)dxdy=E(0).
\end{equation*}

\vspace{3mm}

Regarding the well-posedness for the IVP \eqref{zk-2d} when $k=1$, the best local result available in the literature was recently obtained by Kinoshita \cite{Ki}
for initial data in $H^s(\R^2)$, for $s>-1/4$, and using the conserved quantities above global well-posedness is established
in $H^s(\R^2)$, $s\ge 0$. Previous results include those by  Molinet and Pilod \cite{MP}, Gr\"unrock and Herr \cite{GH}, Faminskii \cite{F} and
Linares and Pastor \cite{LPa2}.

In the case of the  generalized Zakharov-Kuznetsov equation (gZK), i.e. \eqref{zk-2d} for $k\geq 2$,
we shall briefly describe what  is known for the well-posedness of the associated IVP.   In the 2D case, the scale argument suggests local well-posedness results for data in $H^s(\R^2)$, for $s>s_k=1-2/k$.   Sharp local results were obtained by Ribaud and Vento \cite{RV2}, for 
$k\ge 4$.   In \cite{G2}, Gr\"unrock proved the local well-posedness in $\dot{H}^{s_k}(\R^n)$ for $n=2,3$ and 
$s_k=n/2-2/k$. For the particular nonlinearity $k=2$, also called modified Zakharov-Kuznetsov equation (mZK), for which  $L^2(\R^2)$ is the critical space suggested by the scaling argument, local well-posedness  was shown for data in $H^s(\R^2)$, for $s\ge 1/4$ in \cite{RV2}.
Still for the nonlinearity $k=2$ global well-posedness in $H^s(\R^2)$,  $s>53/63$, for initial data with suitable $L^2$ norm, was established by Linares and Pastor in \cite{LPa1}. Recently, this Sobolev index was pushed down to $s>3/4$ by Bhattacharya, Farah, and Roudenko in \cite{BFS}.

Here our main goal is to establish a dispersive blow-up result for solutions of the IVP associated to the ZK  and gZK equations. We recall that the notion of dispersive  singularity was first addressed in \cite{BBM} (see also \cite{BS2}) for solutions of the linear
Korteweg-de Vries equation,
\begin{equation*}
	\p_t u+\p_xu+\p_x^3u=0, \hskip15pt x\in \R, \, t\in\R.
\end{equation*}
Roughly, it is possible to obtain a solution $u=u(x,t)$ that blows up in the $L^{\infty}$-norm in finite time by considering an infinitely smooth, bounded initial wave profile $u(x,0)=u_0(x)$, possessing  finite energy. Moreover, 
the blow-up point $(x^{*},t^{*})$ can be specified arbitrarily in the upper-half plane  $\R\times (0,\infty)$. In \cite{BS1}  Bona and Saut started the mathematical 
analysis of the dispersive blow-up for solutions of the nonlinear generalized KdV equation. More precisely, they proved the following:

\begin{TA}[\cite{BS1}] Let $T>0$ be given and let $\{(x_n,t_n)\}_{n=1}^{\infty}$ be  a sequence of points in $\R\times(0,T)$ without finite limit points and such that $\{t_n\}_{n=1}^{\infty}$ is bounded below by a positive constant. Let either $s=0$ and $k=1$ or $s\ge 2$ and $k\ge 1$ be an arbitrary integer. Then there exists $\psi\in H^s(\R)\cap C^{\infty}(\R)$ such that the solution of the IVP
	\begin{equation*}
		\begin{cases}
			u_t+\p_x^3u+u^k\p_xu=0,\\
			u(x,0)=\psi(x),
		\end{cases}
	\end{equation*}
	satisfies
	\begin{enumerate}
		\item $u$ lies in $L^{\infty}([0,T]: H^s(\R))\cap L^2([0,T] :H^{s+1}_{\rm loc}(\R))$, or in $C([0,T]: H^s(\R))\cap L^2([0,T] :H^{s+1}_{\rm loc}(\R))$, if $s\ge 2$.
		\item $\p_x^su$ is continuous on $\R\times (0,T)\backslash \cup_{n=1}^{\infty}\{(x_n,t_n)\}$, and
		\item $\underset{(x,t)\to (x_n,t_n)} {\lim}  \p_x^s u(x,t)=+\infty$, for $n=1,2,\dots$.
	\end{enumerate}
\end{TA}

The main idea behind the proof is to show that the Duhamel term associated to the solution of the IVP is smoother than the linear term of the
solution.  In \cite{LS} Linares and Scialom obtained the result above, for $k\ge 2$, by employing the smoothing effect properties without using 
weighted Sobolev spaces. Recently, Linares, Ponce and Smith \cite{LPS} improved the previous result using fractional weighted spaces 
in the case $k=1$, i.e., for the KdV equation (see also \cite{LiPal}). 

Analogous phenomena also appear in other linear dispersive equations, such as the linear Schr\"odinger equation and the free surface water waves system linearized around the rest state \cite{BS2}. In \cite{BS2} Bona and Saut constructed initial data with point singularities for solutions of
the linear Schr\"odinger equation. Bona, Ponce, Saut and Sparber \cite{BPSS} established the dispersive blow-up for the semilinear Schr\"odinger equation in dimension $n$ and other Schr\"odinger type equations. Simplifications of the proof and low regularity requirements were given in  \cite{HT}. The main tools employed to show these results were the intrinsic smoothing effects
of these dispersive equations. We remark that this is the only known result  up to date regarding dispersive blow-up in  $n>1$ dimensions.

In order to state and prove our results  we will exploit the symmetric form of \eqref{zk-2d}. We recall
that  the linear change of variables introduced by
Gr\"unrock and Herr in \cite{GH}, i.e.,
\begin{equation*}
	x'=\mu x+\lambda y,\,\,\, y'=\mu x-\lambda y,
\end{equation*}
with $\mu=4^{-1/3}, \lambda=\sqrt 3 \mu$, symmetrizes the equation so that the IVP \eqref{zk-2d} can
be written as
\begin{equation}\label{zks}
	\begin{cases}
		\p_t v+\p_x^3v+\p_y^3v+4^{-1/3}(v^k \p_xv+v^k \p_yv)=0, \hskip5pt (x,y)\in\R^2, t\in\R,\\
		v(x,y,0)=v_0(x,y).
	\end{cases}
\end{equation}

This allows us to consider the IVP \eqref{zks}  instead of \eqref{zk-2d} without changing the well-posedness theory.
One of the reasons to work with this symmetric equation is that we obtain better Strichartz estimates (see Lemma \ref{strich}) which will be key
in our argument. Besides, some technical issues are easier to deal with in the symmetrized equation, as we will explain below.  

The integral form of the solution of \eqref{zks} then becomes  (we drop the unimportant constant $4^{-1/3}$)
\begin{equation}\label{integral}
	v(t)= V(t)v_0-\int_0^t V(t-t') v^k(\p_xv+\p_yv) (t')\,dt',
\end{equation}
where $V(t)$ is the unitary group associated to the symmetric linear problem
\begin{equation}\label{linear-zk}
	\begin{cases}
		\p_t v+\p_x^3v+\p_y^3v=0,\\
		v(x,y,0)=v_0(x,y),
	\end{cases}
\end{equation}
which, via Fourier transform, is given by $\widehat{V(t)v_0}(\xi,\eta)= e^{it(\xi^3+\eta^3)}\widehat{v}_0(\xi,\eta)$.

We now state our main result concerning the existence  of  dispersive blow-up taking place for the ZK  equation.

\begin{theorem}\label{main1} Let $s\in [1,2)$. There exists
	\begin{equation*}
		v_0\in H^s(\R^2)\cap C^{\infty}(\R^2)\cap L^2(\langle x\rangle^r\,dxdy),
	\end{equation*}
	for some $r>0$, so that the solution $v(t)$ of the IVP \eqref{zks}, for $k=1$, is global in time with 
	\begin{equation*}
		v\in C(\R: H^s(\R^2)) \cap \dots
	\end{equation*}
	and it satisfies
	\begin{equation}\label{dispersiveblowup}
		\begin{cases}
			v(t)\in C^1(\R^2),\hskip2.8cm  t>0, t\notin\Z^{+},\\
			v(t) \in C^1(\R^2\backslash {(0,0)})\backslash C^1(\R^2),\hskip15pt t\in\Z^{+}.
		\end{cases}
	\end{equation}
	
	Moreover, if we set
	\begin{equation}\label{duhamel-zk}
	z_1(t) = \int_0^t V(t-t') v \left(\p_x v+\p_y v\right)(t')\, dt',
	\end{equation}
 we have that $z_1(t)$ is in $C^1(\R^2)$ for any $t>0$.
	
\end{theorem}

\begin{remark}  This result shows that the dispersive blow-up phenomenon can be established in higher dimensions for nonlinear models with nonlinearities involving derivatives.
\end{remark}

In order to prove Theorem \ref{main1} we first construct an initial data  whose linear time evolution fails to be $C^1(\R^2)$ at all positive integer 
times,
borrowing some ideas from \cite{LPS}. Then we show that, even though such a initial data is in $H^{2^{-}}(\R^2)$,
the Duhamel term $z_1(t)$
actually is slightly more regular and belongs to $H^{2^{+}}(\R^2) \subset C^1(\R^2)$ for any $t>0$. This smoothing gain in the
integral part encapsulates the dispersive blow-up phenomenon as the singularity development 
is exclusively due to the linear evolution component in \eqref{integral}.

To explain our approach we will consider the nonlinear term  $v\partial_x v$  and then  proceed to estimate it along with its derivatives  in the $x$-variable. To control $2^{+}$ derivatives in $x$ for  $z_1(t)$ in the $L^2$-norm we will first apply the dual version of the smoothing effect of Kato type (see Lemma \ref{seffect}). This
allows us to gain $1$ derivative but we still end up with  the term $D^{1^{+}}_x(v\p_x v)$ in a $L^1_xL^2_{yT}$-norm.  We notice issues are reduced 
to considering the worst term $D_x^{1^+}\partial_x v$ but  in order to avoid the use of the  $L^1_x$ norm we must introduce some weights which bring into play the local well-posedness theory for the  IVP \eqref{zks}  in weighted Sobolev spaces. Recently, Fonseca and Pachon \cite{FP} (see also \cite{BJM})
established the theory in these spaces. Then
our next concern is to estimate $D^{1^{+}}_x\p_xv$. This can be done by using the Strichartz estimates that permit a gain of some extra derivative in
the $x$-variable.  When we analyze the worst term we choose to introduce the commutator  $[D^{1^{+}}_x,v]\p_xv$, which morally is better than the 
term we have already discussed. Nevertheless this commutator  is in the $L^1_xL^2_{yT}$-norm. Again we need to introduce a weight to have this term in 
$L^2_TL^2_{xy}$-norm. At this point we have to appeal to a Kato-Ponce type commutator estimate in weighted spaces proved by Cruz-Uribe and Naibo in \cite{CN}. To control the terms after applying the commutator we will use interpolation estimates introduced by  Nahas-Ponce  \cite{NP} and Ponce \cite{Po} to distribute 
weights and derivatives. The key point here is to use Strichartz estimates and avoid Sobolev embedding. This will close the estimates.
Same arguments work when we look at the $y$-derivatives.

\vspace{3mm}

{\bf Case $k\ge 2$}.  We observe that the analysis described above can be also used to show a similar result as Theorem \ref{main1} for solutions
of the IVP \eqref{zks} for $k\ge 2$. More precisely,

\begin{theorem}\label{main2} Let $s\in [1,2)$ and $k\ge 2$. There exists
	\begin{equation*}
		v_0\in H^s(\R^2)\cap C^{\infty}(\R^2) \cap L^2(\langle x\rangle^r\,dxdy)
	\end{equation*}
	with $\|v_0\|_{H^1}\ll 1$ so that the solution $v(t)$ of the IVP associated to \eqref{zks} is
	global in time with 
	\begin{equation*}
		v\in C(\R: H^s(\R^2)) \cap \dots
	\end{equation*}
	satisfies
	\begin{equation*}
		\begin{cases}
			v( t)\in C^1(\R^2),\hskip3cm  t>0, t\notin\Z^{+},\\
			v( t) \in C^1\left(\R^2\backslash {(0,0)}\right)\backslash C^1(\R^2),\hskip15pt t\in\Z^{+}.
		\end{cases}
	\end{equation*}
\end{theorem}

Since the proof of this result follows by employing a similar analysis as the proof of Theorem \ref{main1} we will not give the details. We point out, however, the fact that the initial data has small $H^1$ norm guarantees that the corresponding solution is indeed global in time (see \cite{FLP} and \cite{LPa2}).

\medskip

To complete our set of results we consider a generalized kind of dispersive blow-up.  We will first show that, given a time $t^{*}\neq 0$, it is possible to construct an initial data
in $H^1(\R^2)\cap W^{1,p}(\R^2)$, $p>2$,  such that the solution of the linear problem associated to the gZK equation
is not in  $W^{1,p}(\R^2)$ at time $t^*$. Moreover, we will prove that the solution of the nonlinear problem inherits such a property but this is still  provided by
the linear part of the solution. The statement will be the first part of our last theorem below.

The second result regards propagation of regularity of solutions of IVP  \eqref{zks}. 
 In \cite{ILP} Isaza, Linares and Ponce, considering suitable solutions of the IVP associated to the  generalized KdV equation
\begin{equation}\label{kdv}
\begin{cases}
\p_tu+\p_x^3u+u^k\p_xu=0, \hskip10pt x, t\in\R, \;k\in \Z^{+},\\
u(x,0)=u_0(x),
\end{cases}
\end{equation}
established the propagation of regularity in the right-hand side of the data for positive times. More precisely,
\begin{TB}[\cite{ILP}] \label{t1}
If  $u_0\in H^{{3/4}^{+}}(\R)$ and for some $\,l\in \Z^{+},\,\;l\geq 1$, and $x_0\in \R$,
\begin{equation*}
\|\,\partial_x^l u_0\|^2_{L^2((x_0,\infty))}=\int_{x_0}^{\infty}|\partial_x^l u_0(x)|^2dx<\infty,
\end{equation*}
then the solution of the IVP \eqref{kdv} provided by the local theory in \cite{KPV} satisfies  that for any $v>0$ and $\epsilon>0$
\begin{equation*}
\underset{0\le t\le T}{\sup}\;\int^{\infty}_{x_0+\epsilon -vt } (\partial_x^j u)^2(x,t)\,dx<c,
\end{equation*}
for $j=0,1, \dots, l$ with $c = c(l; \|u_0\|_{H^{{3/4}^{+}}};\|\,\partial_x^l u_0\|_{L^2((x_0,\infty))} ; v; \epsilon; T)$.

 In particular, for any $\,t\in (0,T]$, the restriction of $u(\cdot, t)$ to any interval $(x_0, \infty)$ belongs to $H^l((x_0,\infty))$.

Moreover, for any $v\geq 0$, $\epsilon>0$ and $R>0$ 
\begin{equation*}
\int_0^T\int_{x_0+\epsilon -vt}^{x_0+R-vt}  (\partial_x^{l+1} u)^2(x,t)\,dx dt< c,
\end{equation*}
where now  $c = c(l; \|u_0\|_{H^{{3/4}^{+}}};\|\,\partial_x^l u_0\|_{L^2((x_0,\infty))} ; v; \epsilon; R; T)$.
\end{TB}

In \cite{LP2} Linares and Ponce extended this result for solutions of the Zakharov-Kusnetzov equation. 
More precisely,

\begin{TC}[\cite{LP2}]\label{3D-regularity}   Let $u_0\in H^s(\R^2)$ with $s>2$. If for some $(a,b)\in \R^2$ satisfying
\begin{equation*}
\label{main-hyp}
a>0,\;\, \;b\geq 0\,\;\;\;\;\;\text{and}\;\;\;\;\;\,\sqrt{3}a>|b|,
\end{equation*}
 and for some $j\in \Z^+,\;j\geq 3$,
\begin{equation*}
\label{mainhyp}
N_{j}\equiv \sum_{|\alpha|=j} \;\int_{P_{\{a,b,c\}}}(\partial^{\alpha}u_0)^2(x,y)\, dxdy <\infty,
\end{equation*}
then the corresponding solution $u=u(x,y,t)$ of the IVP for the ZK equation  provided by the local theory satisfies that
for any $v\geq 0$, $\epsilon>0$ and $R>4\epsilon$
\begin{equation*}
\label{main-result}
\begin{aligned}
&\sup_{0\leq t\leq T} \sum_{|\alpha|\leq j}\int_{P_{\{a,b,c-vt+\epsilon\}}}(\partial^{\alpha}u)^2(x,y,t)\, dxdy\\
&\quad\quad+\sum_{|\alpha|=j+1} \int_0^T \int_{H_{\{a,b,c-vt+\epsilon,c-vt+R\}}} (\partial^{\alpha}u)^2(x,y,t)\, dxdy dt\\
&\leq C=C(\|u_0\|_{H^s};\{N_l:\,1\leq l\leq j\};j;a;b;v;T;\epsilon;R),
\end{aligned}
\end{equation*}
where $P_{\{a,b,c\}}$ denotes the half-space
\begin{equation*}
P_{\{a,b,c\}}=\{(x,y)\in\R^2\,:\, ax+by\geq c\},
\end{equation*}
and $H_{\{a,b,c,d\}}$ denotes the strip 
\begin{equation*}
H_{\{a,b,c,d\}}=\{(x,y)\in\R^2\,:\, c\leq ax+by\leq d\}.
\end{equation*}
\end{TC}

Our next result complement this last one. More precisely, we will construct  initial data in  $H^1(\R^2)\cap W^{r,p}(\R^2)$, for some $r $ and $p$, such that the singularities of 
the solutions do not propagate in any direction. This statement will be the second part of our last theorem below.

We summarize the results described above in the next theorem.


\begin{theorem}\label{main3} 
Fix $k\geq2$.

\begin{enumerate}
\item \label{item-i}Let $t^{*}\neq 0$ be a given real number. Then there exists $v_0\in H^1(\R^2)\cap W^{1,p}(\R^2)$, $p>2$, such that the corresponding solution of the IVP
\eqref{zks} is global in time and satisfies:
\begin{enumerate}
\item $v\in C(\R:\, H^1(\R^2))$;

\indent and

\item $v(t^{*})\notin W^{1,p}(\R^2)$ for every $p>2$.\\
\end{enumerate}
\item\label{item-ii} There exist $r>1$, $p\in(2,10)$, and an initial data
$v_0\in H^1(\R^2)\cap W^{r,p}(\R^2)$
such that the corresponding solution  of the IVP \eqref{zks} is global in time and satisfies:
\begin{enumerate}
	\item $v\in C(\R:\, H^1(\R^2))$;
	
	\indent and

	\item there exists $t>0$ such that
	\begin{equation*}
	v( t)\notin W^{r,p}(\R^2_{+}) \text{\hskip10pt and \hskip10pt} v( -t)\notin W^{r,p}(\R^2_{+}),
	\end{equation*}
	where \begin{equation}\label{thm3a}
	\R^2_{+}:=\{(x,y)\in\R^2\;:\; y\ge 0\}.
	\end{equation}
	\end{enumerate}
The same result holds for $ \R^2_{-}:=\{(x,y)\in\R^2\;:\; y\le 0\}$.
\end{enumerate}
\end{theorem}

To prove this theorem we extend previous analysis introduced in \cite{Li} and \cite{LPS}. The main ingredient is the smoothness of the Duhamel term as before. In this case 
we use pure Sobolev spaces, combined with the smoothing effect of Kato type; the main tool being used. More precisely, we have 

\begin{proposition}\label{smooth-zk} Fix $k\geq2$. Let $v_0\in H^s(\R^2)$, $s= 1, 2, \dots$, and $v \in C([-T,T]:H^s(\R^2))$ the local solution of the IVP \eqref{zks}. Then 
\begin{equation*}
z_k(t)=\int_0^t V(t-t') v^k\left( \partial_x v+\partial_y v\right) (t') \,dt' \in C([-T,T] : H^{s+1}(\R^2)).
\end{equation*}
\end{proposition}

The argument to prove Proposition \ref{smooth-zk} employs the smoothing properties of solutions of the linear problem and the properties of the solution of the IVP itself. We will give the
details only for the case $k=2$. For the other nonlinearities we shall follow the arguments given here and those in \cite{LPa2} where the local theory
for the IVP \eqref{zk-2d} is established in $H^s(\R^2)$, for $s>s_0$, $0<s_0<1$. This latter result implies global well-posedness of the IVP in $H^s(\R^2)$,
$s\ge 1$, whenever the initial data are small enough.

\medskip

The paper is organized as follows. In Section \ref{estimates} we recall estimates and results for the linear and
nonlinear ZK and gZK equations, that will be needed in the proofs that follow. In Section \ref{initialdata} we construct
the smooth initial data whose linear evolution develops singularities at all future integer times and which 
embodies the dispersive blow-up phenomenon. In Section \ref{nonlinearsmoothing1} we prove the nonlinear smoothing
effect for the ZK equation ($k=1$), which in this particular case consists in showing that for initial data in
$H^{2^-}(\R^2)$ the Duhamel term  corresponding to the nonlinear evolution
is actually slightly more regular, it belongs to $H^{2^+}(\R^2)$. Finally, in Section \ref{nonlinearsmoothing2} we prove  Theorem \ref{main3} and
Proposition \ref{smooth-zk}.

%
%
%
%
%

\section{Notation and Preliminary Estimates}\label{estimates}

In this section we introduce some notation used throughout the paper and describe some smoothing properties of solutions of the linear  IVP  associated  to the generalized 
Zakharov-Kuznetsov equation in \eqref{zks}. We use $c$ to denote various constants that may vary line by line. Let $a$ and $b$ be positive real numbers, the
notation $a \lesssim b$ means that there exists a positive constant $c$ such that $a \leq cb$.
Given a real number $r$, we use $r^+$ and $r^-$ to denote, respectively, $r +\varepsilon$ and $r -\varepsilon$ for some $\varepsilon > 0$ sufficiently small.

Let $ q,r \geq1$ and $I\subset\R$ an interval; the mixed norm in the spaces $L^q_{I}L^r_{xy}$  of a function $f=f(x,y,t)$ is defined as
$$
\|f\|_{L^q_{I}L^r_{xy}}=\left(\int_I\|f(\cdot,t)\|^q_{L^r_{xy}}dt\right)^{\frac{1}{q}},
$$
with the usual modifications if either $q=\infty$ or $r=\infty$. If $I=[-T,T]$ we use the notation $\|f\|_{L^q_{T}L^r_{xy}}$. Moreover, if $I=\mathbb{R}$ we shall use  $\|f\|_{L_t^qL^r_{xy}}$. Similar definitions and considerations may be made interchanging the variables $x,y$, and $t$.

Given a function $g=g(x,y)$, its Fourier transform will be denoted by $\widehat{g}$. Given $s\in\R$ and $p\geq1$, $W^{s,p}=W^{s,p}(\R^2)$ denotes the usual Sobolev space. The standard $L^2$-based Sobolev space, that is, when $p=2$, will be denoted by $H^s:=H^s(\R^2)$. For any $s,r\in\R$, $Z_{s,r}$ denotes the weighted Sobolev space
$$
Z_{s,r}:=H^s(\R^2)\cap L^2((1+x^2+y^2)^rdxdy).
$$

For a function $f$ defined on $\R^n$, $n\geq1$, $D^sf$ and $J^sf$ are defined, via Fourier transform, as
$$
\widehat{D^sf}(\zeta)=|\zeta|^s\widehat{f}(\zeta) \quad \mbox{and}\quad
\widehat{J^sf}(\zeta)=\langle \zeta \rangle^s\widehat{f}(\zeta), \quad \zeta\in\R^n,
$$
where, as it is customary, we use the notation $\langle \zeta \rangle^s=(1+ |\zeta|^2)^{s/2}$.
Given a function $f=f(x,y)$ defined on $\R^2$, by $D^s_xf$ and $D^s_yf$ we denote the fractional derivatives of order $s$ with respect to the $x$ and $y$ variables, that is, via Fourier transform, 
$$
\widehat{D_x^sf}(\xi,\eta)=|\xi|^s\widehat{f}(\xi,\eta) \quad \mbox{and}\quad \widehat{D_y^sf}(\xi,\eta)=|\eta|^s\widehat{f}(\xi,\eta).
$$
 We also define  $J^s_x$ and $J^s_y$  as
$$
\widehat{J_x^sf}(\xi,\eta)=\langle \xi \rangle^s\widehat{f}(\xi,\eta), \quad \mbox{and}\quad \widehat{J_y^sf}(\xi,\eta)=\langle \eta \rangle^s\widehat{f}(\xi,\eta).
$$

Next we present the smoothing properties of the solutions of the linear problem \eqref{linear-zk}. We start by recalling the dispersive estimates
\begin{equation}\label{strich-x0}
\|D^{\theta\epsilon}_x V(t)v_0\|_{L^p_{xy}}\leq c|t|^{-\frac{\theta(2+\varepsilon)}{3}}\|v_0\|_{L^{p'}_{xy}}, \quad t\neq0,
\end{equation}
and
\begin{equation}\label{strich-y0}
\|D^{\theta\epsilon}_y V(t)v_0\|_{L^p_{xy}}\leq c|t|^{-\frac{\theta(2+\varepsilon)}{3}}\|v_0\|_{L^{p'}_{xy}},  \quad t\neq0,
\end{equation}
where $0\leq \varepsilon\leq1/2$, $0\leq \theta\leq1$, $p=\frac{2}{1-\theta}$, and $\frac{1}{p}+\frac{1}{p'}=1$ (see \cite{BJM} and \cite{LPa2}). This inequalities imply the
Strichartz estimates for the linear propagator.

\begin{lemma}\label{strich}
The solution of the linear problem \eqref{linear-zk} satisfies
\begin{equation}\label{strich-x}
\|D^{\frac{\theta\epsilon}{2}}_x V(t)v_0\|_{L^q_tL^p_{xy}}\le c\|v_0\|_{L^2_{xy}}
\end{equation}
and 
\begin{equation}\label{strich-y}
\|D^{\frac{\theta\epsilon}{2}}_y V(t)v_0\|_{L^q_tL^p_{xy}}\le c\|v_0\|_{L^2_{xy}},
\end{equation}
for $\epsilon\in [0,1/2]$, $\theta\in [0,1]$, $p=\frac{2}{1-\theta}$, and $q=\frac{6}{\theta(2+\epsilon)}$.
\end{lemma}
\begin{proof}
Inequality \eqref{strich-x} was originally proved for solutions of the linear problem associated with equation \eqref{zk-2d}. However, the proof of the lemma follows the same strategy (see also Section 2 in \cite{BJM}).
\end{proof}

\begin{remark}\label{rem1} Observe that $p=q$ in Lemma \ref{strich} if and only if $\theta=\frac{3}{5+\epsilon}$. More precisely, if and only if $p=q=\frac{2(5+\epsilon)}{2+\epsilon}$.
\end{remark}

As a consequence of Lemma \ref{strich} we obtain.

\begin{corollary}\label{stri0}
For any $0\leq\varepsilon\leq1/2$ we have
$$
\|D_x^{\frac{\varepsilon}{2}}V(t)v_0\|_{L^{12/5}_TL^{\infty}_{xy}}\leq c_T \|v_0\|_{L^2_{xy}}
$$
and
$$
\|D_y^{\frac{\varepsilon}{2}}V(t)v_0\|_{L^{12/5}_TL^{\infty}_{xy}}\leq c_T \|v_0\|_{L^2_{xy}},
$$
where $c_T$ is a constant depending on $T$. 
\end{corollary}
\begin{proof}
See Lemma 2.3 in \cite{LPa1}.
\end{proof}

The smoothing effects of Kato's type are given in the following lemma.
\begin{lemma}\label{seffect}
 The solution of the linear problem \eqref{linear-zk} satisfies
\begin{equation}\label{seffect-x}
\|\nabla V(t)v_0\|_{L^{\infty}_x L^2_{yt}}\le c\|v_0\|_{L^2_{xy}}
\end{equation}
and
\begin{equation}\label{dseffect-x}
\sup_{t\in[-T,T]}\|\nabla \int_0^t V(t-t') g(\cdot,t')\,dt'\|_{L^2_{xy}}\le c\|g\|_{L^1_xL^2_{yT}}.
\end{equation}

Same estimates hold interchanging $y$ in place of $x$.
\end{lemma}
\begin{proof}
See Theorem 2.2 in \cite{F} and Lemma 2.2 in \cite{BJM}. Note that \eqref{dseffect-x} is nothing but the dual version of \eqref{seffect-x}.
\end{proof}

The  maximal function estimates for solution of the linear IVP \eqref{linear-zk} are given in the
next lemma.

\begin{lemma}\label{strich-mzk}
For $s>3/4$ and $T>0$ it holds that
\begin{equation}\label{mfe-x-mzk}
\|V(t)v_0\|_{L^2_xL^{\infty}_{yT}}\le c_T\|v_0\|_{H^s_{xy}}
\end{equation}
and
\begin{equation}\label{mfe-y-mzk}
\|V(t)v_0\|_{L^2_yL^{\infty}_{xT}}\le c_T\|v_0\|_{H^s_{xy}},
\end{equation}
where $c_T$ is a constant depending only on $T$.
\end{lemma}
\begin{proof}
See Theorem 2.4 in \cite{F} and  Lemma 2.3 in \cite{BJM}.
\end{proof}

\subsection{Additional Tools} Next we will list a series of results useful when we treat the nonlinear case for the ZK equation.
In our analysis we will need a Kato-Ponce commutator estimate (\cite{KP})  in weighted spaces. In next lemma $A_p$ denotes the Muckenhoupt class on $\R^n$ (see \cite{CN} for details) and $\mathcal{S}(\R^n)$ stands for the Schwartz space.

 \begin{lemma}\label{cruz}
  Let $1 < p, q < \infty$ and $\frac{1}{2} < r < \infty$ be such that $\frac{1}{r}=\frac{1}{p}+\frac{1}{q}$. If $v\in A_p$,  $w\in A_q$ and $s > \max\{0, n( \frac{1}{r} - 1)\}$ or $s$ is a non-negative even integer, then for all
 $f, g\in \mathcal{S}(\R^n)$,
 \begin{equation*}
 \|D^s(fg) - fD^sg\|_{L^r(v^{\frac{r}{p}}w^{\frac{r}{q}})}
\le c \|D^sf\|_{L^p(v)}\|g\|_{L^q(w)}+\|\nabla f\|_{L^p(v)}\|D^{s-1}g\|_{L^q(w)},
 \end{equation*}
 where $D^s$ represents the homogeneous fractional derivative of order $s$ on $\R^n$.
 \end{lemma}
\begin{proof}
See Theorem 1.1 in \cite{CN}.
\end{proof}

Next we recall a useful Leibniz's rule for fractional derivatives.
\begin{lemma}\label{leibniz} Let $s\in (0,1)$ and $p\in (1,\infty)$. Then
\begin{equation}\label{fraclr}
\|D^{s}(fg)-fD^{s}g-gD^{s}f\|_{L^p(\R)} \le c\,\|g\|_{L^{\infty}(\R)}\|D^{s}f\|_{L^p(\R)},
\end{equation}
\end{lemma}
\begin{proof}
See Theorem A.12 in  \cite{KPV1}.
\end{proof}

We end this section with the following interpolation results, which is valid in any dimension $n\geq1$.
\begin{lemma}\label{commlemma} 
Assume $a,b>0$, $1<p<\infty$, and $\theta\in(0,1)$. If $J^af\in L^p(\R^n)$ and $\langle x \rangle^bf\in L^p(\R^n)$ then
\begin{equation}\label{commlemma1}
\|\langle x \rangle^{(1-\theta)b}J^{\theta a}f\|_{L^p}\le c \|\langle x \rangle^bf\|_{L^p}^{1-\theta}\|J^af\|_{L^p}^\theta.
\end{equation}
The same holds for homogeneous derivatives $D^{a}$ in place of $J^a$.
Moreover, for $p=2$,
\begin{equation}\label{commlemma2}
\|J^{\theta a}(\langle x \rangle^{(1-\theta)b}f)\|_{L^2}\le c \|\langle x \rangle^bf\|_{L^2}^{1-\theta}\|J^a f\|_{L^2}^\theta.
\end{equation}

\end{lemma}
\begin{proof}
The proof of \eqref{commlemma1} is based  on the Three Lines Theorem. We refer to \cite{NP} for the details in the case $p=2$. For general $p$ the proof is exactly the same. Inequality \eqref{commlemma2} is simply Plancherel's theorem applied
to \eqref{commlemma1} when $p=2$.
\end{proof}

\section{Linear singularities}\label{initialdata}

The goal in this section is  to construct the initial data for Theorem \ref{main1}  such that the corresponding linear evolution  exhibits  singularities for future times. More specifically, we explicitly determine $v_0 \in H^s(\R^2)\cap C^{\infty}(\R^2)\cap L^2(\langle x\rangle^s\,dxdy)$, for all $s<2$, and 
show that for all positive integer time instants 
$t \in  \Z^{+}$, $V(t)v_0$ fails to be $C^1(\R^2)$. After we prove in the following section that, for that same initial data, the Duhamel terms  of
the nonlinear solutions remain in $C^1(\R^2)$ for all times, this will serve to establish  that singularities of the  full solutions to the corresponding IVPs for \eqref{zks} occur due to the linear component of the evolutions.

Let us start by considering $v_0\in H^s(\R^2), s<2$, and such that $e^{x+y}v_0\in L^2(\R^2).$ 
If we now define $w(t)=e^{x+y}v(t)$, where $v(t)=V(t)v_0$ is the solution of the linear IVP \eqref{linear-zk}, we
conclude that $w(0)\in L^2(\R^2)$ and $w$ satisfies the PDE:
 \begin{equation}\label{heatslsk}
\p_t w=3\Delta w+2w-3(\p_x w+\p_y w)-(\p_x^3w+\p_y^3w), \hskip5pt (x,y)\in\R^2, t>0.
\end{equation}
Therefore, it follows that $w$ is given by:
\begin{equation}
w(t)=V(t)(e^{3t\Delta}e^{-3t\p_x-3t\p_y}e^{2t}e^{x+y}v_0)
\end{equation} 
and hence
\begin{equation}
\begin{aligned}
e^{x+y}V(t)v_0&=V(t)(e^{3t\Delta}e^{-3t\p_x-3t\p_y}e^{2t}e^{x+y}v_0(x,y))\\
&=V(t)(e^{3t\Delta}e^{x+y-4t}v_0(x-3t,y-3t)).
\end{aligned}
\end{equation}

From Plancherel's theorem and direct estimates of the heat kernel we obtain the following estimates for derivatives of any order $ |\lambda|=m$, $\lambda=(\lambda_1,\lambda_2)$, in $L^{2}(\R^2)$ of $e^{x+y}V(t)v_0, \,t>0$:
\begin{equation}
\begin{aligned}\label{postimes}
&\| \p_{x,y}^\lambda (e^{x+y}V(t)v_0)\|_{L^2}\\
&=\|\p_{x,y}^\lambda V(t)(e^{3t\Delta}e^{x+y-4t}v_0(x-3t,y-3t))\|_{L^2}\\
&\leq\||\xi|^{\lambda_1}\,|\eta|^{\lambda_2}\, e^{-3t\xi^2-3t\eta^2}\|_{L_{\xi,\eta} ^\infty}\|e^{x+y-4t}v_0(x-3t,y-3t)\|_{L^2}\\
&\leq\||\xi|^{\lambda_1}\,\, e^{-3t\xi^2}\|_{L_{\xi} ^\infty}  \| |\eta|^{\lambda_2}\, e^{-3t\eta ^2}\|_{L_{\eta} ^\infty}\|e^{2t+x-3t+y-3t}v_0(x-3t,y-3t)\|_{L^2}\\
&\leq \frac{C_{\lambda_1}}{(\sqrt t)^{\lambda_1}}\frac{C_{\lambda_2}}{(\sqrt t)^{\lambda_2}}e^{2t}\|e^{x+y}v_0\|_{L^2}\\
&\leq \frac{C_\lambda}{t^{\frac m2}}e^{2t}\|e^{x+y}v_0\|_{L^2}.
\end{aligned}
\end{equation}
 
For negative times we define $w(t)=e^{-x-y}v_0$ and assume that $e^{-x-y}v_0\in L^2(\R^2).$
Once we perform a similar estimate to the former one we get for $t<0$:
\begin{equation}\label{negtimes}
\| \p_{x,y}^\lambda (e^{-x-y}V(t)v_0)\|_{L^2}\leq \frac{C_\lambda}{|t|^{\frac{m}{2}}}e^{2|t|}\|e^{-x-y}v_0\|_{L^2}.
\end{equation}

Next, in order to define appropriate initial data $v_0$ that will display the formation of dispersive singularities, we follow the argument in \cite{LPS} and consider $\phi (x,y)=e^{-\sqrt{x^2+y^2}}, \;x,y\in\R$.
It is easy to check that $\phi \in H^s(\R^2)\cap L^2(|(x,y)|^s\,dxdy)$ for any $0<s<2$, therefore $\phi$ is a continuous function but it just misses the regularity to be a $C^1$ function in $\R^2$ at the origin $(x,y)=(0,0)$.
With the help of this $\phi$, the unitary group $V(t)$ and a sequence $(\alpha _j)_j$ of positive real numbers converging to zero sufficiently fast, let us define:
\begin{equation}\label{inidata}
v_0=\sum_{j=1}^\infty \,\alpha_j V(-j)\phi.
\end{equation}

From the known local and global well-posedness theory for the two dimensional ZK equation, discussed in the introduction above,
it follows that for this initial data $v_0$ there exist  globally defined solutions $v(t)\in C(\R;  H^s(\R^2)\cap L^2(|(x,y)|^s\,dxdy))$ of the symmetrized IVP  in \eqref{zks}, with $k=1$, for any $1\leq s<2$ (see also Theorem \ref{lwp-zk} below).

We also claim that $v_0$ in \eqref{inidata} belongs to $C^\infty(\R^2).$  This is equivalent to showing that $e^{-x-y}v_0\in C^\infty(\R^2)$ and from the Sobolev embedding it is enough to prove that $e^{-x-y}v_0\in H^m(\R^2)$ for any $m\in \mathbb N$. For this purpose we restrict our attention now to  \eqref{linear-zk} with  initial datum $\phi$ and note that it satisfies $e^{-x-y}\phi\in L^{2}(\R^2)$ so that, from estimate \eqref{negtimes}, we conclude that for $|\lambda|=|(\lambda_1,\lambda_2|)|=m$:
\begin{equation}   
\begin{aligned}
\| \p_{x,y}^\lambda (e^{-x-y}v_0)\|_{L^2}&\leq \sum_{j=1}^\infty \,\alpha_j \| \p_{x,y}^\lambda (e^{-x-y}V(-j)\phi)\|_{L^2}\\
&\leq \sum_{j=1}^\infty \,\alpha_j\frac{C_\lambda}{j^{\frac{m}{2}}}e^{2j}\|e^{-x-y}\phi\|_{L^2}<\infty.
\end{aligned}
\end{equation}  
for sufficiently small $\alpha_j$'s.

Finally, we consider the linear evolution of the  IVP \eqref{linear-zk} with initial datum $v_0$ and note that for any $t>0, t\notin \mathbb N, V(t)v_0=\sum_{j=1}^\infty \,\alpha_j V(t-j)\phi\in C^{\infty}(\R^2)$. In fact,  we proceed as before and reduce matters to showing that for every $m\in \mathbb N,\; e^{-x-y}V(t)v_0\in H^m(\R^2)$. This claim follows from   \eqref{postimes} and  \eqref{negtimes} since
\begin{equation}
\begin{aligned}
\| \p_{x,y}^\lambda (e^{-x-y}V(t)v_0)\|_{L^2}&\leq \sum_{j=1}^\infty \,\alpha_j \| \p_{x,y}^\lambda (e^{-x-y}V(t-j)\phi)\|_{L^2}\\
&\leq C_\lambda \sum_{j=1}^\infty \,\alpha_j\frac{1}{|t-j|^{\frac{m}{2}}}e^{2|t-j|}<\infty.
\end{aligned}
\end{equation}  

On the other hand, if $t=n\in\mathbb N$, we have
\begin{equation}
V(t)v_0=\alpha_n\phi+\sum_{j=1,j\neq n}^\infty \,\alpha_j V(t-j)\phi,
\end{equation} 
where the series can be shown again to be in $C^{\infty}(\R^2)$ and therefore at these times the linear flow cannot be $C^1(\R^2)$, which we will identify as the dispersive blow-up taking place at $(x,y)=(0,0)$ on
an infinite number of time instants.

\section{Proof of nonlinear smoothing for $k=1$}\label{nonlinearsmoothing1}

We start this section by giving the statement of the local theory of the IVP \eqref{zks} with $k=1$  that will be convenient for our purposes. Recall that in this case, \eqref{zks} writes as
\begin{equation}\label{zk-sim-b}
\begin{cases}
\p_t u+\p_x^3u+\p_y^3u+ u\p_xu+u\p_yu=0, \hskip15pt (x,y)\in \R^2, \; t\in\R,\\
u(x,y,0)=u_0(x,y).
\end{cases}
\end{equation} 

\begin{theorem}\label{lwp-zk} Let $s>3/4$ and $u_0\in Z_{s, s/2}$. Then there exist $T=T(\|u_0\|_{H^s})>0$ and a unique solution $u$ of \eqref{zk-sim-b}
such that
\begin{equation}\label{zk-sim-b1}
u\in C([-T,T]; Z_{s,s/2}),
\end{equation}
\begin{equation}\label{zk-sim-b2}
\|D^s_x\p_x u\|_{L^{\infty}_xL^2_{yT}} +\|D^s_y\p_xu\|_{L^{\infty}_xL^2_{yT}} +\|D^s_x\p_y u\|_{L^{\infty}_yL^2_{xT}} +\|D^s_y\p_y u\|_{L^{\infty}_yL^2_{xT}} <\infty,
\end{equation}
\begin{equation}\label{zk-sim-b3}
\|\p_x u\|_{L^2_TL^{\infty}_{xy}} +\|\p_y u\|_{L^2_TL^{\infty}_{xy}}  <\infty,
\end{equation}
and
\begin{equation}\label{zk-sim-b4}
\|u\|_{L^2_x L^{\infty}_{yT}} +\|u\|_{L^2_yL^{\infty}_{xT}}  <\infty.
\end{equation}

Moreover, for any $T'\in (0,T)$ there exists a neighborhood  $V$ of $u_0\in Z_{s,s/2}$ such that  the data-solution map
$\tilde{u}_0 \mapsto \tilde{u}$ from $V$ into the class defined in \eqref{zk-sim-b1}-\eqref{zk-sim-b4}, with $T'$ instead of $T$,  is Lipschitz.
\end{theorem}
\begin{proof}
	For a detailed proof of this result see \cite{BJM} and \cite{FP}.
\end{proof}


Next, for the initial data constructed in \eqref{inidata}, we will show that $z_1(t)$ defined in \eqref{duhamel-zk} belongs to $H^{2^+}(\R^2)$. This  implies in particular that  $z_1(t)\in C^1(\R^2)$. We will only prove that
\begin{equation}\label{dbu1}
z(t):=\int_0^t V(t-t') u\partial_x u(t')\,dt'
\end{equation}
belongs $H^{2^+}(\R^2)$. This suffices to complete
the proof as the analogous estimates for the Duhamel term associated to the nonlinearity $u\partial_y u$ are exactly
the same, with the roles of $x$ and $y$ exchanged with respect to this case.

We will estimate
\begin{equation*}
\|z(t)\|_{H^{2+a}}^2=\|z(t)\|_{L^2}^2+\|D^{2+a}_xz(t)\|_{L^2}^2+\|D^{2+a}_yz(t)\|_{L^2}^2,
\end{equation*}
for $a\in (0,3/22)$.  Since $z(t)\in L^2(\R^2)$ it is enough to estimate $\|D^{2+a}_xz(t)\|_{L^2}^2$ and $\|D^{2+a}_yz(t)\|_{L^2}^2$.

We start by considering $\|D^{2+a}_xz(t)\|_{L^2}$. Using the dual version of the smoothing effect \eqref{dseffect-x}, it follows that
\begin{equation}\label{dbu2}
\begin{split}
\|D^{2+a}_x z(t)\|_{L^2_{xy}} &= \left\|\p_x \int_0^t V(t-t') D^{1+a}_x (u\p_xu)(t')\,dt'\right\|_{L^2_{xy}}\\
&\lesssim \|D^{1+a}_x (u\p_xu)(t)\|_{L^1_xL^2_{yT}}\\
&\lesssim \|u D^{1+a}_x\p_xu\|_{L^1_xL^2_{yT}}+\|D^{1+a}_x (u\p_xu)-uD^{1+a}_x\p_xu\|_{L^1_xL^2_{yT}}\\
&= N\!L_1+N\!L_2
\end{split}
\end{equation}

Next we estimate each term on the right-hand side of \eqref{dbu2}.





  \subsection{Estimates for ${N\!L}_1$} We will employ the Strichartz estimates \eqref{strich-x} and \eqref{strich-y} to estimate $N\!L_1$. 
  From  Remark \ref{rem1} we choose 
  $r= \frac{2(5+\epsilon)}{2+\epsilon}$. Then  H\"older's inequality,
  Sobolev's embedding and interpolation yield
  \begin{equation}\label{arg-1}
\begin{split}
 \|u D^{1+a}_x\p_xu\|_{L^1_xL^2_{yT}}&\le  \|u\|_{L^{p_1}_xL^{q_1}_{yT}}\| D^{2+a}_x u\|_{L^r_{xyT}}\\
 &\le c\|\ji x\jd^{\alpha}u \|_{L^{q_1}_x L^{q_1}_{yT}}   \|D^{2+a}_x u\|_{L^r_{xyT}}\\
 &\le cT^{1/{q_1}}\|\ji x\jd^{\alpha}u \|_{L^{\infty}_T L^{q_1}_{xy}} \|D^{2+a}_x u\|_{L^r_{xyT}}\\
 &\le cT^{1/{q_1}}\|\ji (x,y)\jd^{\alpha}u \|_{L^{\infty}_T L^{q_1}_{xy}} \|D^{2+a}_x u\|_{L^r_{xyT}}\\
 &\le cT^{1/{q_1}}\|J^{s_r}(\ji (x,y)\jd^{\alpha}u) \|_{L^{\infty}_T L^2_{xy}} \|D^{2+a}_x u\|_{L^r_{xyT}}\\ 
 \end{split}
 \end{equation}
where
\begin{equation*}
\alpha>\frac{1}{2},\hskip10pt  \frac{1}{p_1}=1-\frac{1}{r}, \hskip10pt \frac{1}{q_1}=\frac12-\frac{1}{r}, \hskip10pt s_r\geq2\left(\frac12-\frac{1}{q_1}\right)=\frac{2}{r}.
\end{equation*}

\begin{remark}\label{rem2}
Since $\epsilon \in [0,1/2]$ we have $4<r\le 5$; so	 we may (and will) take $s_r=\frac12$.
\end{remark}

Using Lemma \ref{commlemma}, we have
\begin{equation*}
\|J^{1/2}(\ji (x,y)\jd^{\alpha}u) \|_{L^2_{xy}} 
\le c \|J^{\gamma}u\|^{\lambda}_{L^2_{xy}}\|\ji (x,y)\jd^{\beta} u\|^{(1-\lambda)}_{L^2_{xy}}
\end{equation*}
for  $\gamma \lambda= \frac12$, $\beta(1-\lambda)=\alpha$ and $\lambda\in(0,1)$.  That is, we 
have  $\gamma=\frac12\frac{1}{\lambda}$ and $\beta=\frac{\alpha}{1-\lambda}>\frac12\frac{1}{1-\lambda}$.  Then taking $\lambda=\frac13$ we obtain $\gamma=\frac32$ and $\beta> \frac34$. Thus $2\beta>\frac32$ and
\begin{equation}\label{final}
\begin{split}
 \|u &D^{1+a}_x\p_xu\|_{L^1_xL^2_{yT}}\\
 &\le cT^{1/4} \|J^{3/2}u\|^{1/3}_{L^{\infty}_T L^2_{xy}}\|\ji (x,y)\jd^{3/4^+} u\|^{2/3}_{L^{\infty}_T L^2_{xy}} \|D^{2+a}_x u\|_{L^r_{xyT}}.
 \end{split}
 \end{equation}
The first two terms on the right-hand side of \eqref{final} are bounded by Theorem \ref{lwp-zk}. The third one
requires a more careful analysis which we pursue in the next section.


\medskip

\subsection{Estimate for  $\|D^{2+a}_x u\|_{L^r_{xyT}}$}\label{estimanorma}

Let us write $s_0=1+a-\frac{\epsilon\theta}{2}$ and  observe that $s_0\in (0,1)$ when $a<\frac{\epsilon\theta}{2}=\frac{3\epsilon}{2(5+\epsilon)}$, $\epsilon\in [0,1/2)$. We recall  that the solution of  \eqref{zk-sim-b} is given by
\begin{equation}\label{int-sol}
u(t)= V(t)u_0-\int_0^t V(t-t')\,( u\p_xu +u\p_yu)(t')\,dt',
\end{equation}
for $t\in(-T,T)$, $T>0$ given by the local theory.

\medskip

First we use the Strichartz estimate \eqref{strich-x} to estimate the linear part of \eqref{int-sol} as
\begin{equation*}
\|D^{2+a}_x V(t)u_0\|_{L^r_{xyT}}= \|D^{\frac{\epsilon\theta}{2}}_xV(t) D^{1+s_0}_xu_0\|_{L^r_{xyT}}\le \|D^{2^{-}}_xu_0\|_{L^2_{xy}}
\end{equation*}
because $s_0<1$.
Using this bound, as well as Minkowski's inequality, the Cauchy-Schwarz inequality in time and again \eqref{strich-x}, for the
integral part of \eqref{int-sol}, we obtain
\begin{equation}\label{sr1}
\begin{split}
\|D^{2+a}_x u\|_{L^r_{xyT}}
&\le c\|D^{1+s_0}_xu_0\|_{L^2_{xy}}\\
&\hskip10pt+ cT^{1/2} \|D^{s_0}_x\p_x(u\p_xu)\|_{L^2_{xyT}}+cT^{1/2} \|D^{s_0}_x\p_x(u\p_yu)\|_{L^2_{xyT}}.
\end{split}
\end{equation}

We first consider the second term on the right-hand side of \eqref{sr1}. The usual Leibniz rule yields 
\begin{equation}\label{sr2}
\|D^{s_0}_x\p_x(u\p_xu)\|_{L^2_{xyT}}\le \|D^{s_0}_x(u\p_x^2u)\|_{L^2_{xyT}}+\|D^{s_0}_x(\p_xu\p_xu)\|_{L^2_{xyT}}.
\end{equation}
We analyze next each term on the right-hand side of \eqref{sr2}.  We apply the fractional Leibiniz rule in Lemma \ref{leibniz} in the $x$-direction and then we use
H\"older's inequality to deduce
\begin{equation}\label{sr3}
\begin{split}
& \|D^{s_0}_x(u\p_x^2u)\|_{L^2_{xyT}}\\
&\le c \|D^{s_0}_xu\,\p_x^2u\|_{L^2_{xyT}}+c\|u\,D^{s_0}_x\p_x^2u\|_{L^2_{xyT}}
+c\| \|\p_x^2u\|_{L^{\infty}_{x}}\|D^{s_0}_xu\|_{L^2_{x}}\|_{L^2_{yT}}\\
&\le c\| \|\p_x^2u\|_{L^{\infty}_{xy}}\|D^{s_0}_xu\|_{L^2_{xy}}\|_{L^2_T}+c\|u\|_{L^2_xL^{\infty}_{yT}}\|D^{s_0}_x\p_x^2u\|_{L^{\infty}_xL^2_{yT}}\\
 &\le c\big( \|\p_x^2u\|_{L^2_TL^{\infty}_{xy}}\|D^{s_0}_xu\|_{L^{\infty}_TL^2_{xy}}+\|u\|_{L^2_xL^{\infty}_{yT}}\|D^{s_0}_x\p_x^2u\|_{L^{\infty}_xL^2_{yT}})\\
 \end{split}
 \end{equation}
 From the local theory we can see that all the terms on the right-hand side can be controlled for data in $H^{2^{-}}(\R^2)$ regularity.  Observe that the norm 
 $ \|\p_x^2u\|_{L^2_TL^{\infty}_{xy}}$ is finite since we use the Strichartz estimate \eqref{strich-x} to deduce it and it does not introduce loss of derivatives.
 
 A similar but easier  argument leads to 
 \begin{equation}\label{sr4}
 \|D^{s_0}_x(\p_xu\p_xu)\|_{L^2_{xyT}}\le c\,\|\p_x u\|_{L^2_TL^{\infty}_{xy}}\|D^{s_0+1}_xu\|_{L^{\infty}_TL^2_{xy}},
 \end{equation}
 whose terms on the right-hand side can be estimated using the local theory.

For the last term in \eqref{sr1} we  follow similar arguments as in \eqref{sr3} and \eqref{sr4} to conclude that
\begin{equation}\label{sr5}
\begin{split}
 &\|D^{s_0}_x\p_x(u\p_yu)\|_{L^2_{xyT}}\\
 &\le c\, \big(\|\p_x\p_yu\|_{L^2_TL^{\infty}_{xy}}\|D^{s_0}_xu\|_{L^{\infty}_TL^2_{xy}}+ \|u\|_{L^2_xL^{\infty}_{yT}}\|D^{s_0}_x\p_x\p_yu\|_{L^{\infty}_xL^2_{yT}}\big)\\
 &\hskip10pt+c\,\big(\|\p_x u\|_{L^2_TL^{\infty}_{xy}}\|D^{s_0}_xu\|_{L^{\infty}_TL^2_{xy}}
 +\,\|\p_y u\|_{L^2_TL^{\infty}_{xy}}\|D^{s_0+1}_xu\|_{L^{\infty}_TL^2_{xy}}).
 \end{split}
 \end{equation}
Once again the local theory for data in $H^{2^{-}}(\R^2)$ allows us to deduce that the right-hand side of   \eqref{sr5} is finite.  As we commented above the
norm $\|\p_x\p_yu\|_{L^2_TL^{\infty}_{xy}}$ is finite because the Strichartz estimates do not introduce extra derivative and so we may control it using the arguments
employed to obtain the local theory.

Gathering together all estimates above we conclude that $NL_1$ is finite.

\subsection{Estimates for ${N\!L}_2$} \label{estimacomut}

 Here we will estimate the term 
 $$
NL_2=I:= \|D^{1+a}_x (u\p_xu)-uD^{1+a}_x\p_xu\|_{L^1_xL^2_{yT}}.
 $$
First of all, let us consider $\gamma>1$ to be chosen sufficiently close to 1, as to satisfy all the requirements
that follow. Thus, H\"older's inequality in $x$ variable give
\begin{equation}\label{comm0}
I=\|\langle x\rangle^{-\gamma/2}\|\langle x\rangle^{\gamma/2}[D_x^{1+a},u]\p_xu\|_{L^2_{yT}}\|_{L^1_x}\lesssim \|\langle x\rangle^{\gamma/2}[D_x^{1+a},u]\p_xu\|_{L^2_{xyT}},
\end{equation}
where we used that $\langle x\rangle^{-\gamma/2}\in L^2(\R)$. Define $v(x)=w(x)=\langle x\rangle^{\gamma}$. Then, 
$$
v^{2/4}w^{2/4}=\langle x\rangle^{\gamma/2}\langle x\rangle^{\gamma/2}=\langle x\rangle^{\gamma}.
$$
In view of Lemma \ref{cruz}, with $r=2$ and $p=q=4$,
\begin{equation}\label{comm1}
\begin{split}
\| &\langle x\rangle^{\gamma/2}[D_x^{1+a},u]\p_xu\|_{L^2_{x}}=\|[D_x^{1+a},u]\p_xu\|_{L^2_{x}(\langle x\rangle^{\gamma})}\\
&\lesssim \|\langle x\rangle^{\gamma/4}D_x^{1+a}u\|_{L^4_x}\|\langle x\rangle^{\gamma/4}\p_xu\|_{L^4_x} +\|\langle x\rangle^{\gamma/4}\p_xu\|_{L^4_x}\|\langle x\rangle^{\gamma/4}D_x^{a}\p_xu\|_{L^4_x}.
\end{split}
\end{equation}
By using that $\mathcal{H}^2=-1$, where $\mathcal{H}$ is the Hilbert transform, and $\langle x\rangle^{\gamma}$ belongs to $A_4$ class, we have
\begin{equation}\label{hilbert}
\|\langle x\rangle^{\gamma/4} D_x^{a}\p_x u\|_{L^4_x}=\|\langle x\rangle^{\gamma/4}\mathcal{H} D_x^{a+1}u\|_{L^4_x}\lesssim \|\langle x\rangle^{\gamma/4}D_x^{a+1}u\|_{L^4_x}.
\end{equation}
Consequently, from \eqref{comm1} and Young's inequality,
\begin{equation}\label{comm2}
\begin{split}
\|\langle x\rangle^{\gamma/2}[D_x^{1+a},u]\p_xu\|_{L^2_{x}}
&\lesssim \|\langle x\rangle^{\gamma/4} D_x^{a+1}u\|_{L^4_x} \|\langle x\rangle^{\gamma/4}\p_xu\|_{L^4_x} \\
&\lesssim \|\langle x\rangle^{\gamma/4}D_x^{a+1}u\|_{L^4_x}^2+\|\langle x\rangle^{\gamma/4}\p_xu\|_{L^4_x}^2.
\end{split}
\end{equation}

Substituting \eqref{comm2} in \eqref{comm0},
\begin{equation}\label{comm3}
\begin{split}
I&\lesssim \|\langle x\rangle^{\gamma/4} D_x^{a+1}u\|_{L^4_{Txy}}^2+\|\langle x\rangle^{\gamma/4}\p_xu\|_{L^4_{Txy}}^2\\
&=: I_1+I_2.
\end{split}
\end{equation}

Let us estimate $I_1$. We start by using \eqref{commlemma1}   in $L^4$ and with homogeneous derivatives, to split the norm
into two, one of which only has weights and the other only has derivatives
\begin{equation}
\begin{split}
\|\langle x\rangle^{\gamma/4}D_x^{a+1}u\|_{L^4_{x}}&
\lesssim   \|\langle x\rangle^{\gamma/4\lambda}u\|_{L^4_{x}}^{\lambda}
\| D_x^{(a+1)/(1-\lambda)}u\|_{L^4_{x}}^{1-\lambda} \\
& \lesssim  \|\langle x\rangle^{\gamma/4\lambda}u\|_{L^4_{x}} + 
\| D_x^{(a+1)/(1-\lambda)}u\|_{L^4_{x}},
\end{split}
\end{equation}
with $\lambda \in (0,1)$, and then
\begin{equation}
\begin{split}
\|\langle x\rangle^{\gamma/4}D_x^{a+1}u\|_{L^4_{xy}} &\lesssim
\|\langle x\rangle^{\gamma/4\lambda}u\|_{L^4_{xy}} + 
\| D_x^{(a+1)/(1-\lambda)}u\|_{L^4_{xy}}\\
&\lesssim
\|\langle (x,y)\rangle^{\gamma/4\lambda}u\|_{L^4_{xy}} + 
\| D_x^{(a+1)/(1-\lambda)}u\|_{L^4_{xy}}.
\end{split}
\end{equation}
For the first of these terms we use Sobolev's embedding and Lemma \ref{commlemma} again 
\begin{equation}
\begin{split}
\|\langle (x,y)\rangle^{\gamma/4\lambda}u\|_{L^4_{xy}}&
\lesssim   \|J^{1/2}\left( \langle (x,y)\rangle^{\gamma/4\lambda}u\right) \|_{L^2_{xy}}
\\
& \lesssim  \|\langle (x,y)\rangle^{\gamma/4 \lambda \sigma}  u  \|_{L^2_{xy}}^{\sigma}
\|J^{1/2(1-\sigma)}u\|_{L^2_{xy}}^{1-\sigma} \\
& \lesssim \|\langle (x,y)\rangle^{\gamma/4 \lambda \sigma}  u  \|_{L^2_{xy}}+
\|J^{1/2(1-\sigma)}u\|_{L^2_{xy}}.
\end{split}
\end{equation}
So we choose $\sigma=\frac{3}{4}-\varepsilon$ and $\lambda=\frac13+\delta$, with $\varepsilon$ and $\delta$ small
so that
$$\frac{1}{2(1-\sigma)}=\frac{2}{1+4\varepsilon}<2,$$
and
$$\frac{\gamma}{4 \lambda \sigma}=\frac{\gamma}{(\frac13+\delta)(3-4\varepsilon)}<1,$$
for $\gamma$ close to 1. So, gathering all these estimates back into $I_1$ we get
\begin{equation}
I_1\lesssim T^{1/2}\|\langle (x,y)\rangle^{1^-}  u  \|^2_{L^{\infty}_T L^2_{xy}}+
 T^{1/2}\|J^{2^-}u\|^2_{L^{\infty}_T L^2_{xy}}+\| D_x^{(a+1)\frac{3}{2}^+}u\|^2_{L^4_{Txy}}.
\end{equation} 
The first two terms on the right-hand side of this inequality are bounded by the estimates from the local theory. As
for the third term, we observe that
$$\| D_x^{(a+1)\frac{3}{2}^+}u\|_{L^4_{Txy}}\leq T^{1/12}\| D_x^{(a+1)\frac{3}{2}^+}u\|_{L^{6}_TL^4_{xy}},$$
and that $L^{6}_TL^4_{xy}$ is a Strichartz admissible norm, corresponding to $\theta=1/2$ and $\epsilon=0$ in
Lemma \ref{strich}. So that the linear term in \eqref{int-sol} in this norm will be bounded by
$\| D_x^{(a+1)\frac{3}{2}^+}u_0\|_{L^2_{xy}},$
and following an argument analogous to that in Section \ref{estimanorma}, 
the boundedness of the norm of the full nonlinear solution is guaranteed from the local existence theorem.

As for the term $I_2$ in \eqref{comm3} we start by
switching from the classical derivative to a homogeneous one, in an analogous manner to what was done in \eqref{hilbert},
$$
\|\langle x\rangle^{\gamma/4} \p_x u\|_{L^4_x}=\|\langle x\rangle^{\gamma/4}\mathcal{H} D_x u\|_{L^4_x}\lesssim \|\langle x\rangle^{\gamma/4}D_x u\|_{L^4_x}.
$$
And now we just have to observe that this term $\|\langle x\rangle^{\gamma/4}D_x u\|_{L^4_x}$ is almost
the same as in $I_1$, except that the number of derivatives is $1$, instead of the slightly bigger $1+a$, making the treatment of $I_2$
exactly the same as was done above for $I_1$.

\vspace{5mm}
Now we estimate derivatives in the $y$ variable. Following \eqref{dbu2} we have
that the dual version of the smoothing effect (now in the $y$ variable) yields
\begin{equation}\label{fdy-0a}
\begin{split}
\|D^{2+a}_y&\int_0^t V(t-t') u\p_x u (t')\,dt'\|_{L^2_{xy}} \le \| D^{1+a}_y(u\p_x u)\|_{L^1_yL^2_{xT}}\\
&\le \|u D^{1+a}_y \p_x u\|_{L^1_yL^2_{xT}}+\| D^{1+a}_y(u\p_xu)- uD^{1+a}_y\p_xu\|_{L^1_yL^2_{xT}}\\
&= K_1+K_2.
\end{split}
\end{equation}

Next we estimate $K_1$. Following the argument in \eqref{arg-1} it follows that we
reach an estimate analogous to \eqref{final}:
\begin{equation}\label{fdy-0b}
\begin{split}
 \|u D^{1+a}_y&\p_xu\|_{L^1_xL^2_{yT}}\\
 &\le cT^{1/4} \|J^{3/2}u\|^{1/3}_{L^{\infty}_T L^2_{xy}}\|\ji (x,y)\jd^{3/4^+} u\|^{2/3}_{L^{\infty}_T L^2_{xy}} \|D^{1+a}_y \p_x u\|_{L^r_{xyT}},
 \end{split}
\end{equation}
the only difference being the term $\|D^{1+a}_y \p_x u\|_{L^r_{xyT}}$, with mixed derivatives,
instead of $\|D^{2+a}_x  u\|_{L^r_{xyT}}$ that occurred in \eqref{final} with all $2+a$ derivatives in $x$ only.
However, the same type of argument as in Section \ref{estimanorma} allows us to again control this term,
with a gain of $\frac{\epsilon\theta}{2}=\frac{3\epsilon}{2(5+\epsilon)}\sim \frac{3}{22}$ derivatives in
the $x$ direction,
from the local theory norms. 

On the other hand, following the same ideas as in Section \ref{estimacomut} to estimate the commutator term,
\begin{equation}\label{fdy-0c}
K_2 \le \|\ji y\jd^{\gamma/2} [D^{1+a}_y, u]\p_x u\|_{L^2_{xyT}},
\end{equation}
and using the estimate in Lemma \ref{cruz} again we obtain
\begin{equation}\label{fdy-0d}
\begin{split}
\|\ji y\jd^{\gamma/2} [D^{1+a}_y, u]\p_x u\|_{L^2_y} &\le \|\ji y\jd^{\g/4}\p_x u\|_{L^4_y}  \|\ji y\jd^{\g/4} D^{1+a}_y u\|_{L^4_y}\\
&\hskip10pt +  \|\ji y\jd^{\g/4}\p_y u\|_{L^4_y}  \|\ji y\jd^{\g/4}D^a_y\p_x u\|_{L^4_y}.
\end{split}
\end{equation}

The terms   $\|\ji y\jd^{\g/4} D^{1+a}_y u\|_{L^4_y}$ and  $\|\ji y\jd^{\g/4}\p_y u\|_{L^4_y}$  can be handled exactly as was done in Section \ref{estimacomut}, now in the $y$ direction.  The only novelty here are
the terms  $\|\ji y\jd^{\g/4}D^a_y\p_x u\|_{L^4_y}$  and $\|\ji y\jd^{\g/4}\p_x u\|_{L^4_y}$, with
mixed directions for derivatives and norms. For the second of these terms, we observe that
\begin{equation}\label{fdy-0e}
\|\ji y\jd^{\g/4}\p_x u\|_{L^4_{xy}}\lesssim \|\ji (x,y)\jd^{\g/4}\p_x u\|_{L^4_{xy}}\lesssim
\|\ji (x,y)\jd^{\g/4}D_x u\|_{L^4_{xy}}.
\end{equation}
Hence the argument employed above to estimate $I_2$, which in turn was analogous to the one used to
estimate $I_1$, applies in this case too. 

We are thus only left with controlling $\|\ji y\jd^{\g/4}D^a_y\p_x u\|_{L^4_y}$. But again
 using the same ideas as in the estimates for $I_1$, splitting the norm
by Lemma \ref{commlemma}, and ultimately using Strichartz estimates as in Section \ref{estimanorma}
to recover some derivatives in order to bring the total to $2^-$ derivatives in $L^2$, one can follow similar arguments as before to obtain the required bounds. This finishes the proof.

%
%
%


\section{Proof of nonlinear smoothing for $k=2$}\label{nonlinearsmoothing2}

We now consider solutions of the IVP associated to the modified Zakharov-Kuznetsov equation,
\begin{equation}\label{mzk}
\begin{cases}
\p_t u +\p_x^3u+\p_y^3u+u^2(\p_x u+\p_yu)=0, \hskip15pt (x,y)\in\R^2, \;t\in\R,\\
u(x,y,0)=u_0(x,y).
\end{cases}
\end{equation}

In this section we show that we can obtain the nonlinear smoothing effect of
the Duhamel term of the solution without using weighted Sobolev spaces. We start
by giving the local  well-posedness result that we will use in this case. 



\begin{theorem}\label{lwp-mzk} For any $u_0\in H^s(\R^2)$, $s>3/4$, there exist $T=T(\|u_0\|_{H^s})>0$ and a unique solution $u$ of the IVP \eqref{mzk}
defined in the interval $[-T,T]$ such that
\begin{equation}\label{mzk-1} 
u\in C([-T,T] :H^s(\R^2)),
\end{equation}
\begin{equation}\label{mzk-2} 
\|D^s_x\p_xu\|_{L^{\infty}_xL^2_{yT}}+\|D^s_y\p_xu\|_{L^{\infty}_xL^2_{yT}}
+\|D^s_x\p_yu\|_{L^{\infty}_yL^2_{xT}}+\|D^s_y\p_yu\|_{L^{\infty}_yL^2_{xT}} <\infty,
\end{equation}
\begin{equation}\label{mzk-3} 
\|u\|_{L^3_TL^{\infty}_{xy}}+ \|\p_xu\|_{L^{12/5}_TL^{\infty}_{xy}}+\|\p_yu\|_{L^{12/5}_TL^{\infty}_{xy}}<\infty,
\end{equation}
and
\begin{equation}\label{mzk-4} 
\|u\|_{L^2_xL^{\infty}_{yT}} +\|u\|_{L^2_yL^{\infty}_{xT}} <\infty.
\end{equation}

Moreover, for any $T'\in (0,T)$ there exists a neighborhood $\mathcal{W}$ of $u_0\in H^s(\R^2)$ such that the map $\tilde{u}_0\mapsto \tilde{u}(t)$ from
$\mathcal{W}$ into the the class defined by \eqref{mzk-1}--\eqref{mzk-4}, with $T'$ instead of $T$, is Lipschitz.
\end{theorem}
\begin{proof}
	The proof relies on the contraction mapping principle and it is quite similar to that of Theorem 1.1 in \cite{LPa2}. The Strichartz estimates in Corollary \ref{stri0}, Kato's smoothing effects in Lemma \ref{seffect} combined with the maximal function estimate in Lemma \ref{strich-mzk} are enough to recover the loss of derivatives in the nonlinear term of \eqref{mzk}.
\end{proof}

Next we prove Proposition \ref{smooth-zk} in the case $k=2$. As we already said, the general case follows similar arguments.

\begin{proof}[Proof of Proposition \ref{smooth-zk}]
We will prove the proposition in the case $s=1$. For larger values of $s$ the strategy is the same. Thus, it suffices to assume that  $u_0\in H^1(\R^2)$ and prove that
\begin{equation*}
z_2(t)=\int_0^t V(t-t') u^2\left(\p_xu+\p_yu\right)(t')\,dt' \in H^2(\R^2) \hskip15pt \text{for}\hskip10pt t\in[-T,T].
\end{equation*}
Since $z_2(t)$ is in $L^2(\R^2)$, it is enough to show that $\partial_x^2z_2(t)$ and $\partial_y^2z_2(t)$ are in $L^2(\R^2)$ for $t\in[-T,T]$. 

In addition, due to the symmetry with respect to $x$ and $y$ in $z_2(t)$ and in our estimates presented in Section \ref{estimates} we will only show that
\begin{equation}\label{equiH2}
\p_x^2\int_0^t V(t-t') u^2\p_xu(t')\,dt' \quad \mbox{and } \quad  \p_y^2\int_0^t V(t-t') u^2\p_xu(t')\,dt'
\end{equation}
belong to $L^2(\R^2)$. The estimates for the nonlinear term involving $u^2\p_yu$ follow in a similar fashion.

Now using the dual version of the smoothing effect \eqref{dseffect-x},
\begin{equation}\label{dx-1}
\begin{split}
\|\p_x^2\int_0^t &V(t-t') u^2\p_xu(t')\,dt' \|_{L^2_{xy}}\\
 &\lesssim \|\p_x(u^2\p_xu)\|_{L^1_xL^2_{yT}}\\
&\lesssim \|u^2\p_x^2u\|_{L^1_xL^2_{yT}} +\|u (\p_xu)^2\|_{L^1_xL^2_{yT}}\\
&= A_1 +A_2.
\end{split}
\end{equation} 

Applying H\"older's inequality it follows that
\begin{equation}\label{dx-2}
A_1\lesssim \|u\|_{L^2_xL^{\infty}_{yT}}^2\|\p_x^2u\|_{L^{\infty}_xL^2_{yT}},
\end{equation}
which is finite from the local theory in \eqref{mzk-2} and \eqref{mzk-4}.


Again using the H\"older inequality we get
\begin{equation}\label{dx-3}
\begin{split}
A_2 &\lesssim \|u\|_{L^2_xL^{\infty}_{yT}}\|\p_xu\|_{L^4_xL^4_{yT}}^2.
\end{split}
\end{equation}
Since the norm $\|\p_xu\|_{L^4_xL^4_{yT}}$ can be interpolated from the norms $\|\p_x^2u\|_{L^{\infty}_xL^2_{yT}}$ 
in \eqref{mzk-2} and $\|u\|_{L^2_xL^{\infty}_{yT}}$ in \eqref{mzk-4}, which are bounded by Theorem \ref{lwp-mzk}, it follows that
that the term $A_2$ is also finite.

To estimate the second term in \eqref{equiH2} we use again the smoothing effect \eqref{dseffect-x}, this time for $\p_y$, to obtain
\begin{equation}\label{dy-1}
\begin{split}
& \|\p_y^2 \int_0^t V(t-t') \big (u^2\partial_x u)(t')\, dt'\|_{L^2_{xy}}\\
&\lesssim \|u^2\p_y\p_x u\|_{L^1_xL^2_{yT}}+\|u\p_yu\p_xu\|_{L^1_xL^2_{yT}}\\
&= B_1+B_2.
\end{split}
\end{equation}

For the first term, $B_1$, we use H\"older's inequality to get
\begin{equation}\label{dy-2}
B_1 \lesssim \|u\|_{L^2_xL^{\infty}_{yT}}^2\|\p_y\p_xu\|_{L^{\infty}_xL^2_{yT}},
\end{equation}
which is finite by \eqref{mzk-2} and \eqref{mzk-4}.

Applying once more the H\"older inequality, now for $B_2$, we obtain
\begin{equation}\label{dy-3}
\begin{split}
B_2 &\lesssim \|u\|_{L^2_xL^{\infty}_{yT}}\|\p_yu\p_xu\|_{L^2_xL^2_{yT}}\\
&\lesssim  T^{1/12}\|u\|_{L^2_xL^{\infty}_{yT}}\|\p_yu\|_{L^{12/5}_TL^{\infty}_{xy}}\|\p_xu\|_{L^{\infty}_TL^2_{xy}}.
\end{split}
\end{equation}
All the terms in the last inequality are bounded by Theorem \ref{lwp-mzk}.

Combining the above inequalities the result follows.
\end{proof}

\begin{proof}[Proof of Theorem \ref{main3}]

We start proving part \eqref{item-i}. Recall that for any initial data $v_0\in H^1(\R^2)$, as long as  the local solution of \eqref{zks} exists, it  can be written as
\begin{equation}\label{duhamel}
v(t)=V(t)v_0-\int_0^t V(t-t') v^k(\partial_xv+\partial_yv) (t')\, dt'=: V(t) v_0+ z_k(t).
\end{equation}
In addition, from Proposition \ref{smooth-zk} it follows that $z_k(t)\in H^2(\R^2)$. Hence the Sobolev embedding  theorem  implies that $z_k(t)\in W^{1,p}(\R^2)$, for any  $p\ge 2$, as long as it is well defined. 
Then we have to show that there exists $v_0\in H^1(\R^2)\cap W^{1,p}(\R^2)$ such that $U(t^{*})v_0 \notin W^{1,p}(\R^2)$ for every $p>2$.

Let $\phi\in H^1(\R^2)\cap W^{1,1}(\R^2)$ such that $\phi\notin W^{1,p}(\R^2)$ for every $p>2$.  We note that, from \eqref{strich-x0},
\begin{equation}\label{decay-1}
\|V(t)\partial_x \phi\|_{L^{\infty}_{xy}}\le \frac{c}{|t|^{2/3}}\|\partial_x \phi\|_{L^1_{xy}},  \quad t\neq0.
\end{equation}
Also, since $V(t)$ is a unitary group on $L^2(\R^2)$, we have
\begin{equation}\label{decay-2}
\|V(t)\partial_x \phi\|_{L^2_{xy}}= c\|\partial_x \phi\|_{L^2_{xy}}.
\end{equation}
Since $\phi\in H^1(\R^2)\cap W^{1,1}(\R^2)$, an interpolation using \eqref{decay-1} and \eqref{decay-2} gives that $V(t)\p_x\phi\in L^p(\R^2)$, for any $p\geq2$.
Similar estimates hold if we replace $\p_x\phi$ by $\phi$ or $\partial_y \phi$ in \eqref{decay-1} and \eqref{decay-2}, from which we deduce that $V(t)\phi\in W^{1,p}(\R^2)$, for any $p\geq2$ and each $t\neq0$.

Next we choose $v_0(x,y)=V(-t^{*})\phi(x,y)$. Multiplying $v_0$ by a small constant, if necessary, we may assume that the corresponding solution, say, $v(t)$, is global in time with $v\in C(\R:H^1(\R^2))$. Since $V(t^*)$ is a unitary group in $H^1(\R^2)$ we deduce that $v_0\in H^1(\R^2)\cap W^{1,p}(\R^2)$ with $V(t^{*})v_0=\phi \in H^1(\R^2)$ but $\phi\notin W^{1,p}(\R^2)$ for every 
$p>2$. This implies  claim \eqref{item-i}.

\bigskip

Next we prove  part \eqref{item-ii}. First we note that by taking $\varepsilon=1/2$ and $\theta\in[0,1)$ in Lemma \ref{strich}, we have, for $p>2$,
\begin{equation}\label{strich-x-mzk-a}
\|D^{\frac{p-2}{4p}}_x V(t)v_0\|_{L^{\frac{12p}{5(p-2)}}_tL^p_{xy}}\le c\|v_0\|_{L^2_{xy}}
\end{equation}
and
\begin{equation}\label{strich-x-mzk-b}
\|D^{\frac{p-2}{4p}}_y V(t)v_0\|_{L^{\frac{12p}{5(p-2)}}_tL^p_{xy}}\le c\|v_0\|_{L^2_{xy}},
\end{equation}
for any $v_0\in L^2(\R^2)$.

We take $\widetilde{v}_0 \in H^1 (\R^2)$ with $1= r-\frac{p-2}{4p}=r-\hat{p}$ with $\widetilde{v}_0\notin W^{r,p}(\R^2_{+})$,
where  $\R^2_{+}$ is defined as \eqref{thm3a}. From \eqref{strich-x-mzk-a} and \eqref{strich-x-mzk-b} there exists $\tilde{t}>0$ such that 
\begin{equation*}
V(\pm \tilde{t}\,)\widetilde{v}_0,\hskip10pt V(\pm 2\tilde{t}\,)\widetilde{v}_0 \in W^{r,p}(\R^2) \hskip10pt \text{with}\;\; r=1+\frac{p-2}{4p}=j.
\end{equation*}

Next we consider the initial data
\begin{equation}
v_0 = V(\tilde{t}\,)\widetilde{v}_0 +V(-\tilde{t}\,)\widetilde{v}_0.
\end{equation}
We observe that $v_0\in H^1(\R^2)\cap W^{r,p}(\R^2)$. In addition, multiplying $v_0$ by a small constant, if necessary, we may also assume that the corresponding solution  is global in time with $v\in C(\R:H^1(\R^2))$.  On the other hand,  we have from Proposition \ref{smooth-zk} that $z_k(t)\in C([-T,T]:H^2(\R^2))$. Hence, the Sobolev embedding implies that
\begin{equation*}
C([-T,T]:H^2(\R^2)) \hookrightarrow C([-T,T]: W^{r,p}(\R^2))
\end{equation*}
whenever  $p\in(2,10)$.

We finally note that 
\begin{equation*}
V(\tilde{t}\,) v_0= V(2\tilde{t}\,)\widetilde{v}_0+ \widetilde{v}_0 \notin W^{r,p}(\R^2_{+})
\end{equation*}
and the same holds for  $V(-\tilde{t}\,) v_0$. 

Combining the above information we derive the desired result.
\end{proof}

\vskip1cm

\begin{center}
	{\bf Acknowledgments}
\end{center}
The authors would like to thank German Fonseca for fruitful conversations that lead to the end of this project. J. Drumond Silva would like to thank Sim\~ao Correia for useful discussions and the kind hospitality of IMPA, Instituto de Matem\'atica Pura e Aplicada, in
	Rio de Janeiro, and of IMECC, Instituto de Matem\'atica, Estat\'istica e Computa\c{c}\~ao Cient\'ifica of
	UNICAMP, Universidade Estadual de Campinas, where part of this work was developed. 
	This work was developed in the framework and partially supported by the CAPES-FCT convenium {\it Equa\c{c}\~oes de evolu\c{c}\~ao dispersivas}. J. Drumond Silva was partially supported by FCT/Portugal through UID/MAT/04459/2019. F. Linares was partially supported by CNPq and FAPERJ/Brazil. A. Pastor was partially supported by CNPq/Brazil grants 402849/2016-7 and 303098/2016-3.


\vspace{1cm}
\begin{thebibliography}{99}

\bibitem{BBM} T.B. Benjamin, J.L. Bona, J.J. Mahony, {\em Model equations for long waves in nonlinear, dispersive
media}, Philos. Trans. Royal Soc. London, Ser. A 272 (1972), 47--78.

\bibitem{BFS} D. Bhattacharya, L.G. Farah, S. Roudenko, {\em Global well-posedness for low regularity data in the 2D modified Zakharov-Kuznetov equation}, 
arXiv:1906.05822v1.

\bibitem{BPSS} J.L. Bona, G. Ponce, J.-C. Saut,  C. Sparber, {\em Dispersive blow-up for nonlinear Schr\"odinger equations revisited}, 
J. Math. Pures Appl. 102 (2014), 782--811. 

 \bibitem{BS1} J.L. Bona, J.-C. Saut, {\em Dispersive blow up of solutions of generalized KdV equations}, J. Differential Equations 103 (1993), 3--57. 
 
 \bibitem{BS2} J.L. Bona, J.-C. Saut, {\em Dispersive blow up II. Schr\"odinger-type equations, optical and oceanic rogue waves}, Chin. Ann. Math., Ser. B 31 (2010), 793--810.
 
 \bibitem{BJM} E. Bustamante, J. Jimenez, J. Mejia, {\em The Zakharov-Kuznetzov equation in weighted Sobolev spaces}, J. Math. Anal. Appl. 
 433 (2016), 149--175.
 
 \bibitem{CN}  D. Cruz-Uribe, V. Naibo, {\em Kato-Ponce inequalities on weighted and variable Lebesgue space}, Differential Integral Equations 29 (2016), 801--836. 
 
 \bibitem{npv} E. Di Nezza, G. Palatucci, E. Valdinocci, {\em Hitchhiker's guide to the fractional Sobolev spaces} Bull.  Sci.  Math. 136 (2012),  no.  5, 521--573.
 
 \bibitem{F} A.V. Faminskii, {\em The Cauchy problem for the Zakharov-Kuznetsov equation}, Differ. Equ. 31 (1995), 1002--1012. 
 
 \bibitem{FLP}    L.G. Farah, F. Linares, A.  Pastor, {\em  A note on the 2D generalized Zakharov-Kuznetsov equation: local, global and scattering results}, J. Differential Equations 253 (2011), 2558--2571.
 
 \bibitem{FP} G. Fonseca,  M.  Pach\'on, {\em Well-posedness for the two dimensional generalized Zakharov-Kuznetsov equation in anisotropic weighted Sobolev spaces},  J. Math. Anal. Appl. 443 (2016), 566--584.
 
 \bibitem{G2}    A. Gr\"unrock, {\em  On the generalized Zakharov-Kuznetsov equation at critical regularity}. arXiv:1509.09146v1.
 
 
 \bibitem{GH}  A. Gr\"unrock, S. Herr, {\em The Fourier restriction norm method for the Zakharov-Kuznetsov equation},  Discrete Contin. Dyn. Syst. 34 (2014), 2061--2068. 
 
 \bibitem{HT} Y. Hong, M. Taskovic, {\em On dispersive blow-ups for the nonlinear Schr\"odinger equation}, Differential Integral Equations, 29 (2016), 875--888.
 
 \bibitem{H} D. Han-Kwan, {\em  From Vlasov-Poisson to Korteweg-de Vries and Zakharov-Kuznetsov}. Comm. Math. Phys. 324 (2013),  961--993.
 
 \bibitem{ILP}
\newblock  P. Isaza, F.  Linares, G. Ponce,
\newblock \emph{On the propagation of regularity and decay of solutions to the k-generalized Korteweg-de Vries equation},
\newblock  {Comm. Partial Differential Equations}  40 (2015) 1336-1364. 




 
 \bibitem{Ka} T. Kato,  {\em On the Cauchy problem for the (generalized) Korteweg-de Vries equation}. Advances in Mathematics Supplementary Studies, Studies in Applied Math. 8 (1983), 93--128.
 
 
 \bibitem{KP} T. Kato, G. Ponce, {\em Commutator estimates and the Euler and Navier-Stokes equations}, Commun. Pure Appl. Math. 41 (1988), 891--907. 
 
 \bibitem{KPV}  C.E. Kenig, G. Ponce, L. Vega,  {\em Well-posedness of the initial value problem for the Korteweg-de Vries equation}, J. Amer. Math. Soc. 4 (1991), 323--347.
 
\bibitem{KPV1} C.E. Kenig, G. Ponce, L. Vega, {\em Well-posedness and scattering results for the generalized Korteweg--de Vries equation via the contraction principle}, Commun. Pure Appl. Math. 46 (1993), 527--620. 


\bibitem{Ki}  S. Kinoshita, {\em Global Well-posedness for the Cauchy problem of the Zakharov-Kuznetsov equation in 2D}, arXiv:1905.01490.

\bibitem{LLS}    D. Lannes, F.  Linares, J-C. Saut, {\em  The Cauchy problem for the Euler-Poisson system and derivation of the Zakharov-Kuznetsov equation}, Prog. Nonlinear Diff. Eqs Appl. 84 (2013), 181--213.

 \bibitem{Li}  F. Linares,  {\em A higher order modified Korteweg-de Vries equation}, Mat. Apl. Comput. 14 (1995), 253--267.
 
 \bibitem{LiPal} F. Linares, J.M. Palacios, {\em Dispersive blow-up and persistence properties for the Schr\"odingerÐKortewegÐde Vries system}, Nonlinearity, 32 (2019), 4996--5016
 

\bibitem{LPa2} F. Linares, A. Pastor, {\em Well-posedness for the two-dimensional modified Zakharov-Kuznetsov equation},
SIAM J. Math. Anal.  41 (2009), 1323--1339.

\bibitem{LPa1} F. Linares, A. Pastor, {\em Local and global well-posedness for the 2D generalized Zakharov-Kuznetsov equation},
J. Functional Analysis  260 (2011), 1060--1085.



\bibitem{LP} F. Linares, G. Ponce, {\em Introduction to Nonlinear Dispersive Equations}, 2nd Ed. Springer-Verlag, New York, 2015.

\bibitem{LP2}  F. Linares, G. Ponce, {\em On Special Regularity Properties of  Solutions of the Zakharov-Kuznetsov Equation}, Commun. Pure Appl. Anal. 17 (2018), 1561--1572.

\bibitem{LPS}  F. Linares, G. Ponce, D. Smith, {\em On the regularity of solutions to a class of nonlinear dispersive equations}, Math. Ann. 369 (2017), 797--837.

\bibitem{LS} F. Linares, M. Scialom, {\em On the smoothing properties of solutions to the modified Korteweg-de Vries equation}, J. Differential Equations 106 (1993), 141--154. 




\bibitem{MP}  L. Molinet, D. Pilod, {\em Bilinear Strichartz estimates for the Zakharov-Kuznetsov equation and applications}, Ann. Inst. H. Poincar\'e Anal. Non Lin\'eaire 32 (2015), 347--371.


\bibitem{NP}  J. Nahas, G. Ponce, {\em  On the persistent properties of solutions to semi-linear Schr\"odinger equation}, Comm. Partial Differential Equations 34 (2009), 1208--1227. 
 
\bibitem{Po} G. Ponce, {\em Personal Communication}.

\bibitem{RV2}      F. Ribaud, S. Vento, {\em A note on the Cauchy problem for the 2D generalized Zakharov-Kuznetsov equations}, C. R. Acad. Sci. Paris 350 (2012),  499--503.



\bibitem{KZ}   V.E. Zakharov, E.A.   Kuznetsov, {\em On three dimensional solitons}, Sov. Phys. JETP. 39 (1974), 285--286. 
 
 \end{thebibliography}
\end{document}